\documentclass[12pt]{amsart}
\usepackage{amssymb,latexsym}
\usepackage{amsfonts}
\usepackage{amsmath}
\usepackage[colorlinks,linkcolor=blue,anchorcolor=blue,citecolor=blue]{hyperref}

\newcommand\C{{\mathbb C}}

\newcommand\Q{{\mathbb Q}}
\newcommand\R{{\mathbb R}}
\newcommand\N{{\mathbb N}}
\newcommand\Z{{\mathbb Z}}

\newcommand\Res{\mathrm{Res}}

\newcommand\al{\alpha}
\newcommand\be{\beta}
\newcommand\ga{\gamma}

\newtheorem{theorem}{Theorem}[section]
\newtheorem{lemma}[theorem]{Lemma}

\newtheorem{corollary}[theorem]{Corollary}
\newtheorem{proposition}[theorem]{Proposition}

\theoremstyle{remark}

\numberwithin{equation}{section}

\begin{document}

\title[Counting degenerate polynomials]{Counting degenerate polynomials of fixed degree and bounded height}


\author{Art\= uras Dubickas}
\address{Department of Mathematics and Informatics, Vilnius University, Naugarduko 24,
LT-03225 Vilnius, Lithuania}
\email{arturas.dubickas@mif.vu.lt}

\author{Min Sha}
\address{School of Mathematics and Statistics, University of New South Wales,
 Sydney NSW 2052, Australia}
\email{shamin2010@gmail.com}






\begin{abstract}
In this paper, we give sharp upper and lower bounds for the number of degenerate monic (and arbitrary, not necessarily monic) polynomials with integer coefficients of fixed degree $n \ge 2$ and height bounded by $H \ge 2$. The polynomial is called {\em degenerate} if
it has two distinct roots whose quotient is a root of unity. In particular, our bounds imply that non-degenerate linear recurrence sequences can be generated randomly.
\end{abstract}

\maketitle



\section{Introduction}

Recall that every linear recurrence sequence $s_0,s_1,s_2,\dots$ of order $n\ge 2$ is defined by the linear relation
\begin{equation}\label{sequence}
s_{k+n}=a_1s_{k+n-1}+\cdots+a_ns_k \quad (k=0,1,2,\dots).
\end{equation}
 Here, we suppose that the coefficients $a_1,\dots,a_n$ and the initial values $s_0,\dots,s_{n-1}$ of the sequence are some elements of a number field $K$, where $a_n\ne 0$ and $s_j \ne 0$ for at least one
 $j$ in the range
$0 \le j \le n-1$.
The characteristic polynomial of this linear recurrence sequence is
$$
f(X)=X^n-a_1X^{n-1}-\cdots-a_n \in K[X].
$$

The sequence \eqref{sequence} is called \emph{degenerate} if $f$ has a pair of distinct roots whose quotient is a root of unity; otherwise the sequence \eqref{sequence} is called \emph{non-degenerate}. It is
well-known that the sequence \eqref{sequence} may have infinitely many zero terms only if it is degenerate, whereas the non-degenerate sequences contain only finitely many zero terms. See, for instance, \cite[Section 2.1]{EvdPSW} for more details and more references.

Since non-degenerate linear recurrence sequences have much more applications in practice, it is important to investigate how often they occur. If we choose $a_1, \dots, a_n$ as rational integers (or rational num\-bers), the results of this paper imply that almost every linear recurrence sequence generated randomly is non-degenerate; see Theorem \ref{main}. Combining with \cite[Theorem 1.1]{Dubickas2014}, we see that almost every randomly generated linear recurrence sequence with integer coefficients is non-degenerate and has a dominant root, that is, it is exactly what we usually prefer it to be.

By adopting this terminology, we say that a polynomial $f \in \C[X]$ is {\it degenerate\/} if
it has a pair of distinct roots whose quotient is a root of unity. Note that there already exist some methods for testing whether a given polynomial with integer coefficients is degenerate or not, see, e.g.,  \cite{Cipu2011}.
In the sequel, all the polynomials we consider have integer coefficients except in few cases when this will be indicated explicitly.

We first define the set $S_n(H)$ of degenerate monic polynomials with integer coefficients of degree $n\ge 2$ and of height at most $H$, that is,
\begin{equation*}
\begin{split}
S_n(H) =  \{f(X) = X^n &+a_1X^{n-1} + \dots + a_n\in \Z[X]~:\\
&f\ \text{is degenerate},\  |a_j| \le H, \ j =1, \ldots, n\}.
\end{split}
\end{equation*}
Throughout, the height of a polynomial is defined to be the largest modulus of its coefficients. Then, we define
$$
D_n(H)=|S_n(H)|,
$$
where $|A|$ is the cardinality of the set $A$.

To state our results we shall use the following standard notation.
Throughout the paper, we use the Landau symbol $O$ and the Vinogradov symbol $\ll$. Recall that the assertions $U=O(V)$ and $U \ll V$ (sometimes we will write this also as $V \gg U$) are both equivalent to the inequality $|U|\le CV$ with some constant $C>0$.
In this paper, the constants implied in the symbols $O, \ll$ and in the phrase ``up to some constant'' only depend on the degree $n$; also, when we say ``finitely many'', this means that the exact quantity also depends on the degree $n$ only. In the sequel, we always assume that $H$ is a positive integer (greater than $1$ if there is the factor $\log H$ in the corresponding formula), and $n$ is an integer greater than 1.

The main result of this paper is the following:

\begin{theorem}\label{main}
For any integers $n\ge 2$ and $H\ge 2$, we have the following sharp bounds for $D_n(H)$:
$$H \ll D_2(H) \ll H,$$
$$H\log H \ll D_3(H) \ll H\log H,$$
$$\qquad\qquad H^{n-2} \ll D_n(H) \ll H^{n-2} \quad (n\ge 4).$$
\end{theorem}

When $n$ is fixed, Theorem \ref{main} says that the proportion of degenerate monic polynomials among the monic polynomials of degree $n$ and height at most $H$ tends to zero as $H\to \infty$. So, degenerate polynomials cannot be efficiently constructed by a random generation. It is easy to construct reducible degenerate polynomials. But it is not easy to construct irreducible degenerate polynomials of high degree $n$ other than polynomials of the form $g(X^{\ell})$ or $\prod_{j=1}^{\ell} g(\xi_j X)$ with
irreducible (and satisfying some further restrictions) $g \in \Z[X]$ of degree $n/\ell$, where $\xi_1,\dots,\xi_{\ell}$ are the conjugates of a root of unity $\xi_1$ of degree $\ell$.
Several other examples can be found in \cite{Cipu2011}.

In order to obtain lower and upper bounds on the
number $D_n(H)$, we define $I_n(H)$ and $R_n(H)$ as the numbers of irreducible and reducible polynomials $f \in S_n(H)$, respectively, so that
$$
D_n(H)=I_n(H)+R_n(H).
$$

In fact, by an explicit construction, it is quite easy to obtain the claimed lower bound for $D_n(H)$, whereas all the difficulties lie in getting the claimed upper bound for $D_n(H)$. Our approach is to first estimate $I_n(H)$ and $R_n(H)$, respectively, and then sum them up.

Now, we state the following sharp estimates for $I_n(H)$ and $R_n(H)$, and, furthermore, we discover an interesting phenomenon.

\begin{theorem}\label{InH}
Let $n\ge 2$ and $H\ge 1$ be two integers, and let $p$ be the smallest prime divisor of $n$. Then
$$H^{n/p}\ll I_n(H) \ll H^{n/p}.$$
\end{theorem}

For $R_n(H)$, our bounds are the following:

\begin{theorem}\label{RnH}
For integers $n\ge 2$ and $H\ge 2$, we have
$$R_2(H)=\lfloor \sqrt{H} \rfloor,$$
$$H\log H \ll R_3(H) \ll H\log H,$$
$$\qquad\qquad H^{n-2} \ll R_n(H) \ll H^{n-2} \quad (n\ge 4).$$
\end{theorem}

From Theorems \ref{InH} and \ref{RnH}, for each
$n\ge 5$,
we have $HI_n(H)\ll R_n(H)$. Roughly speaking, most of the degenerate monic polynomials of degree $n\ge 5$ are reducible. The reason why this is of interest is that the proportion of reducible monic polynomials of degree $n\ge 2$ among all the monic polynomials of degree $n$ and height at most $H$ tends to zero as $H\to \infty$. However, in the set $S_n(H)$, $n\ge 5$, the proportion of irreducible polynomials tends to zero as $H\to \infty$.

It is easy to see that Theorem~\ref{main} is a direct corollary of Theorems \ref{InH} and  \ref{RnH}. We will prove Theorems \ref{InH} and  \ref{RnH} in Sections \ref{InHH} and \ref{RnHH}, respectively.

Actually, in a similar way we can count the number of degenerate
(not necessarily monic)  polynomials  with integer coefficients of fixed degree and bounded height.
Consider the set
\begin{equation*}
\begin{split}
S_n^{*}(H) =  \{f(X) = &a_0X^n  +a_1X^{n-1} + \dots + a_n\in \Z[X]~:\\
&a_0\ne 0, f\ \text{is degenerate},  |a_j| \le H, j=0,1,\dots,n\}.
\end{split}
\end{equation*}
Put $D_n^{*}(H)=|S_n^{*}(H)|$. As above, define $I_n^{*}(H)$ and $R_n^{*}(H)$ as the numbers of irreducible and reducible polynomials $f \in S_n^{*}(H)$, respectively, so that
$$
D_n^{*}(H)=I_n^{*}(H)+R_n^{*}(H).
$$

To make these definitions precise, we recall that,
in general, a non-zero non-unit element of a commutative ring $R$
is said to be {\it irreducible} if it is not a product of two non-units and {\it reducible} otherwise.
In this paper, the polynomials in $\Z[x]$ are called irreducible if they are irreducible in the ring $\Q[x]$. Note that there is a slight difference between the irreducibility of a polynomial with integer coefficients in the rings $\Z[x]$
and $\Q[x]$.  For example, any linear polynomial is irreducible in $\Q[x]$, but, e.g., $ax+a$ with integer $a \ge 2$, is reducible in $\Z[x]$, whereas the polynomial $x^2-1$ is reducible in both rings $\Z[x]$ and $\Q[x]$.

The analogue of Theorem~\ref{main} can be stated as follows:
\begin{theorem}\label{mainn}
For any integers $n\ge 2$ and $H\ge 2$, we have
$$H^2 \ll D_2^{*}(H) \ll H^2,$$
$$H^2\log H \ll D_3^{*}(H) \ll H^2\log H,$$
$$\qquad\qquad H^{n-1} \ll D_n^{*}(H) \ll H^{n-1} \quad (n\ge 4).$$
\end{theorem}

Comparing Theorems~\ref{main} with \ref{mainn}, we see that
in Theorem~\ref{mainn} there is an extra factor of $H$. This phenomenon
occurs naturally in this kind of problems: compare, for instance,
\cite{Chela1963} (monic case) with \cite{mnn} and \cite{Kuba} (general case), or \cite{mnn0} (monic case) with \cite{iss} (both cases; see also
\cite{iss1}). In Section
\ref{DnH*} we shall state Theorems~\ref{InH*} and \ref{RnH*} that
are
analogues of Theorems~\ref{InH} and \ref{RnH} with an extra
factor of $H$ (except for reducible quadratic polynomials when the bounds in monic and arbitrary cases are different, $\sqrt{H}$ and $H \log H$, respectively).

In the next section we give some preliminaries. In Sections~\ref{InHH}
and \ref{RnHH}
we obtain sharp bounds for
$I_n(H)$ (proving Theorem~\ref{InH}) and for
$R_n(H)$ (proving Theorem~\ref{RnH}), respectively. In Section~\ref{DnH*}
we will complete the proof of Theorem~\ref{mainn}. Finally,
in Section~\ref{comment} we will give some explicit formulas and
evaluate some constants in our formulas for polynomials of
low degree.

\section{Preliminaries}\label{preliminary}

In this section, we gather some preliminaries that will be useful later on.

Given a polynomial $$f(X)=a_0X^n+a_1X^{n-1}+\cdots+a_n=a_0 (X-\al_1)\cdots (X-\al_n) \in \C[x],$$ where $a_0 \ne 0$, its {\it height} is defined by $H(f)=\max_{0 \leq j \leq n} |a_j|$,
and its {\it Mahler measure} by
$$
M(f)=|a_0| \prod_{j=1}^n \max\{1,|\al_j|\}.
$$
For each $f \in \C[x]$ of degree $n$, these quantities are related by the following well-known inequality
$$
H(f) 2^{-n} \leq M(f) \leq H(f) \sqrt{n+1}, \notag
$$
for instance, see \cite[(3.12)]{Waldschmidt2000}. So, for fixed $n$, we have
\begin{equation}\label{Mahler}
H(f) \ll M(f) \ll H(f).
\end{equation}

Let $\rho_k(n,H)$ be the number of monic polynomials
$$
f(X) = X^n +a_1X^{n-1} + \dots + a_n\in \Z[X], \quad n\ge 2,
$$
which are reducible in $\Z[X]$ with an irreducible factor of degree $k$,
$1 \le k \le n/2$, satisfying
$$
H(f)\le H.
$$
In \cite{Waerden1936}, van der Waerden  proved the following sharp lower and upper bounds for $\rho_k(n,H)$; see also \cite{Chela1963}.

\begin{lemma}\label{reducible}
For integers $n\ge 3$ and $k\ge 1$, we have
$$H^{n-k}\ll \rho_k(n,H) \ll H^{n-k} \quad \textrm{if} \quad 1 \le k < n/2,$$
$$H^{n-k}\log H\ll \rho_k(n,H) \ll H^{n-k}\log H \quad \textrm{if} \quad  k=n/2.$$
\end{lemma}

This yields the following corollary.

\begin{corollary}\label{reducible-1}
The number of reducible monic integer polynomials of degree $n \ge 3$ and height at most $H$ is $O(H^{n-1})$.
\end{corollary}

The asymptotical formula for the number of such polynomials was given by Chela in \cite{Chela1963}.

Similarly, if $\rho_k^{*}(n,H)$ is the number of polynomials
$$
f(X) = a_0X^n +a_1X^{n-1} + \dots + a_n\in \Z[X], \quad a_0\ne 0, n\ge 2,
$$
which are reducible in $\Q[X]$ with an irreducible factor of degree $k$, $1 \le k< n/2$, and satisfy $H(f)\le H$, then, by Theorem 4 of \cite{Kuba}, we have the following:

\begin{lemma}\label{reducible2}
For integers $n \ge 3$ and $k\ge 1$, we have
$$H^{n-k+1}\ll \rho_k^{*}(n,H) \ll H^{n-k+1} \quad \textrm{if} \quad 1 \le k<n/2.$$ Furhermore, the number of reducible integer polynomials of degree $n \ge 3$ and height at most $H$ is $O(H^{n})$.
\end{lemma}

The asymptotical formula for the number of reducible integer polynomials
of degree $n \ge 2$ and height at most $H$ with a quite complicated constant was recently given in \cite{mnn}
(see \cite[Example 266]{pol}, \cite{dor} and \cite{Kuba} for some previous bounds in this problem). Many results and asymptotic formulas counting algebraic numbers of fixed degree and bounded heights (but other than naive height) have been obtained in \cite{chern} (Mahler measure), \cite{bar3}, \cite{bar4} (multiplicative height) and in some very recent papers \cite{bar1}, \cite{bar2}.

Some special forms of polynomials will play an important role in getting lower and upper bounds.
The one below is nontrivial. It was obtained by Ferguson \cite{Ferguson1997}; see also the previous result of Boyd \cite{boyd}.

\begin{lemma}\label{Ferguson}
If $f \in \Z[X]$ is an irreducible polynomial which has $m$ roots on a circle $|z|=c$, at least one of which is real, then one has $f(X)=g(X^m)$, where the polynomial $g \in \Z[X]$ has at most one real root on any circle in the plane.
\end{lemma}

The following lemma concerning the upper bound of the moduli of roots of polynomials is a classical result due to Cauchy \cite{Cauchy1} (see also \cite[Theorem 2.5.1]{Mignotte} or \cite[Theorem 1.1.3]{Prasolov}).

\begin{lemma}\label{Cauchy}
All the roots of the polynomial of degree $n \ge 1$
$$
f(X)=a_0X^n+a_1X^{n-1}+\cdots+a_n \in \C[X],
$$
where $a_0\ne 0$ and $(a_1,\dots,a_n)\ne (0,\dots,0),$
are contained in the disc $|z|\le R$, where $R$ is the unique positive solution of the equation
\begin{equation}\label{Cauchy2}
|a_0|X^n-|a_1|X^{n-1}-\cdots-|a_{n-1}|X-|a_n|=0.
\end{equation}
In particular, when $f \in \R[X]$ is the left-hand side of \eqref{Cauchy2}, then
$X=R$ is the unique positive root of $f$.
\end{lemma}

This lemma will assist us in constructing irreducible degenerate polynomials explicitly.

We also need the next result about the resultant of the polynomials $f(X)$ and $f(\eta X)$, where
$\eta \ne 1$.

\begin{lemma}\label{resultant}
Let $f(X)=a_0X^n+a_1X^{n-1}+\dots+a_n$ be a polynomial of degree $n\ge 2$ and with unknown integer coefficients, and let $\eta \ne 1$ be a complex number. Assume that $f(X)\ne f(\eta X)$. Then, for any
fixed $a_0,a_1,\dots, a_{n-1}$, the resultant $\Res(f(X), f(\eta X))$ is a non-zero polynomial in $a_n$ of degree at most $n$.
\end{lemma}

\begin{proof}
Assume that $x$ is a common root of $f(X)$ and $f(\eta X)$. Then, $x$ is also a root of $f(X)-f(\eta X)$. Since $f(X)-f(\eta X)$ is a non-zero polynomial, there are at most $n$ values of $x$ for which $f(x)-f(\eta x)=0$.

Now, let us fix $a_0,a_1,\dots, a_{n-1}$. Suppose that $\Res(f(X), f(\eta X))$ is zero identically. Then, for any $a_n\in\Z$, $f(X)$ and $f(\eta X)$ have common roots. Since there are at most $n$ values for those possible common roots, there exist $a_n, a_n^{\prime}\in\Z$, $a_n\ne a_n^{\prime}$, such that the two polynomials $a_0X^n+a_1X^{n-1}+\dots+a_{n-1}X+a_n$ and $a_0X^n+a_1X^{n-1}+\dots+a_{n-1}X+a_n^{\prime}$ have common roots. But this is possible only if $a_n= a_n^{\prime}$, which leads to a contradiction.
In addition, by the definition of the resultant by a Sylvester matrix, it is easy to see that the degree of the polynomial $\Res(f(X), f(\eta X))$ in $a_n$ is at most $n$.
\end{proof}

This gives the next upper bounds for $D_n(H)$ (and $D_n^{*}(H)$), which (although they are not sharp for $n>2$) will be useful afterwards.

\begin{proposition}\label{trivial}
For each integer $n\ge 2$, we have $D_n(H)=O(H^{n-1})$
and $D_n^{*}(H)=O(H^n)$.
\end{proposition}

\begin{proof}
Notice that, for fixed $n\ge 2$, there are only finitely many
roots of unity which are ratios of two algebraic numbers of
degree at most $n$. (See, e.g., \cite{isacs} or \cite[Corollary 1.3]{DruDub} for a more precise
result asserting that $\deg(\al/\al') \le \deg \al$ whenever $\al, \al'$
are two conjugate algebraic numbers whose quotient $\al/\al'$ is a root of unity.)

If $f(X) = a_0X^n +a_1X^{n-1} + \cdots + a_n\in\Z[X]$ is degenerate, then there exists a root of unity $\eta\ne 1$ which is the ratio of two distinct roots of $f$ (note that there are only finitely many values for such $\eta$). So, $f(X)$ and $f(\eta X)$
have a common root. Hence, the resultant
$$
R_\eta(a_0,a_1,\dots, a_n) =\Res(f(X), f(\eta X))
$$
vanishes. Furthermore, viewing $R_{\eta}(a_0,a_1,\dots, a_n)$ as a polynomial with respect to $a_0,a_1,\dots,a_n$, every point $(a_0,a_1,\dots,a_n)$ at which $R_{\eta}$ vanishes corresponds to a degenerate polynomial $f$.

Clearly, the number of such polynomials $f$ with $a_{n-1}=0$ is $O(H^{n})$. Now, let us assume that $a_{n-1}\ne 0$. Then $f(X)\ne f(\eta X)$. By Lemma \ref{resultant}, if we fix $a_0,a_1,\dots, a_{n-1}$, $R_\eta(a_0,a_1,\ldots, a_n)$ is a non-zero polynomial in $a_n$ of degree at most $n$. Thus, there are at most $n$ values of $a_n$. Hence, the equation $R_\eta(a_0,a_1,\dots, a_n)= 0$ has $O(H^{n})$ integer solutions in variables $a_0,a_1,\dots,a_n$ with $|a_0|,|a_1|, \dots, |a_n| \le H$.
Finally, by summing up the number of these integer solutions running through all possible values of $\eta$, we get $D_n^{*}(H)=O(H^{n})$.

In the monic case, the value of $a_0$ is fixed, $a_0=1$, so, for
$D_n(H)$, we obtain $D_n(H)=O(H^{n-1})$.
\end{proof}

\section{Sharp bounds for $I_n(H)$}\label{InHH}

In this section, we will obtain sharp bounds for $I_n(H)$ (this is perhaps
the most difficult part of the paper).

Let $f\in \Z[X]$ be an irreducible polynomial of degree $n \ge 2$. We say that two of its roots $\al$ and $\be$ belong to the same equivalence class if their quotient is a root of unity.
Suppose that there are $s(f)$ distinct equivalence classes. Using the Galois action on the roots of $f$, it is easy to see that each
equivalence class contains the same number of elements, say, $\ell(f)$ roots of $f$. Thus, as in \cite{Arturas2012} (see also \cite{Schinzel}), we have
\begin{equation}\label{equivalence}
n=s(f)\ell(f).
\end{equation}

Clearly, any irreducible polynomial $f\in\Z[X]$ is degenerate if and only if $\ell(f)\ge 2$. From now on, we will concentrate on monic polynomials and then explain in Section~\ref{DnH*} the differences in the general (non-monic) case.

For integers $n\ge 2, \ell\ge 2$, we define the set
$$
S_{n,\ell}(H)=\{f\in S_n(H): \textrm{$f$ is irreducible, }\ell(f)=\ell \}.
$$
Then, putting $I_{n,\ell}(H)=|S_{n,\ell}(H)|$, we have
$$
I_n(H)=\sum\limits_{\ell=2}^{n} I_{n,\ell}(H).
$$

Evidently, by \eqref{equivalence}, we have $I_{n,\ell}(H)=0$ when $\ell$
does not divide $n$. If $n$ is an integer greater than $1$ and $\ell|n$, we can get the following sharp bounds for $I_{n,\ell}(H)$:

\begin{proposition}\label{Inl}
For any integer $n \ge 2$ and its any divisor $\ell \ge 2$, we have
$$
H^{n/\ell}\ll I_{n,\ell}(H) \ll H^{n/\ell}.
$$
\end{proposition}
\begin{proof}
Take any polynomial $f\in S_{n,\ell}(H)$.  Denote, for brevity, $\ell(f)=\ell$ and $s(f)=s$. Suppose that the $s$ equivalence classes of roots of $f$ are $C_1,\dots,C_s$. Let $\be_j$ be the product of all the elements of $C_j$, $1\le j\le s$. Then, every automorphism of the Galois group of
$\text{Gal}(K/\Q)$, where $K$ is the splitting field of $f$, which maps an element of $C_i$ into an element of $C_j$,  maps a root of unity into a root of unity, so it maps all $\ell$ elements
of $C_i$ into $\ell$ elements of $C_j$. Thus,
$\be_1, \dots, \be_s$ are conjugate algebraic numbers. Hence,
\begin{equation}\label{ggg}
g(X)=(X-\be_1)\cdots (X-\be_s)\in \Z[X].
\end{equation}

We claim that there are only finitely many polynomials $f$ corresponding to the same $g$. Indeed, for $1\le j\le s$, since the quotient of any two elements of $C_j$ is a root of unity, we have $\be_j=\al_j^{\ell}\eta_j$ for some $\al_j \in C_j$ and some root of unity $\eta_j$, where the degree of $\eta_j$ is bounded above by a constant depending only on $n$. This implies that $\be_1,\dots,\be_s$ are distinct. Now, for every fixed $\be_j$, where $1\le j\le s$, there are only finitely many possible $\eta_j$, so there are also finitely many $\al_j$.  Thus, if we fix $g$ (that is, fix $\be_1, \dots, \be_s$), then there are only finitely many of such representatives in each equivalence class $\al_1 \in C_1,\dots,\al_s \in C_s$. Since other roots of $f$ have the form $\al_j \xi_i$ with $1 \le j \le s$ and a root of unity $\xi_i$ of degree at most $n$, there are  only finitely many such degenerate polynomials $f \in \Z[X]$. This completes the proof of the claim.

Note that $M(f)=M(g)$, so
$$
H(f)\ll H(g)\ll H(f)
$$
in view of \eqref{Mahler}.
It follows that the number of such polynomials $f \in S_{n,\ell}(H)$ is bounded above by the number of those corresponding monic polynomials $g \in \Z[X]$ of degree $s=n/\ell$ and height at most $H$ (up to some constant), which is $O(H^{n/\ell})$. This proves the upper bound $I_{n,\ell}(H) \ll H^{n/\ell}$.

We remark that, for $n$ odd (but only for $n$ odd!), another proof of the inequality $I_{n,\ell}(H) \ll H^{n/\ell}$
can be given, by applying
Lemma \ref{Ferguson}. Indeed,
take a polynomial $f\in S_{n,\ell}(H)$. Since $n$ is odd, $f$ has at least one real root. Pick one equivalence class $C$ which has been defined above such that there is a real root $\al$ of $f$ contained in $C$. Notice that these $\ell$ roots in $C$ all have the same modulus $|\al|$.
Assume that this circle contains $t \ge 1$ equivalence classes.
(In general, $t$ can be greater than $1$.) Then there are exactly
$t\ell$ conjugates of $\al$ lying on the circle $|z|=|\al|$ and $\al$ is a real conjugate lying on $|z|=|\al|$.
Hence, by Lemma \ref{Ferguson}, we must have $f(X)=g(X^{t\ell})$ for some $g\in\Z[X]$.  In particular, this means that the coefficient for any term $X^k$ of the polynomial $f$ with $\ell\nmid k$ is zero. So, as above, we obtain $I_{n,\ell}(H)\ll H^{n/\ell}$ (for $n$ odd).

For the lower bound, without restriction of generality we may assume that $H \ge 4$. Consider monic polynomials $g$ of degree $m=n/\ell$
of the form
\begin{equation}\label{ggg1}
g(X)=X^m-2b_1X^{m-1}-\dots-2b_{m-1}X-2(2b_m-1),
\end{equation}
where $1 \le b_1,\dots,b_{m-1} \le \lfloor H/2\rfloor$ and $1\le b_m \le \lfloor H/4\rfloor$. There are at least $\lfloor H/4\rfloor^m$
of such polynomials $g$. They are all irreducible, by Eisenstein's criterion with respect to the prime $2$. Note that if $g$ is as above, then the polynomial $f(X)=g(X^{\ell}) \in \Z[X]$ has degree $n$ and height $H(f)=H(g) \le H$. To complete the proof of the lower bound of the proposition, it remains to show that $f \in S_{n,\ell}(H)$.

By Lemma \ref{Cauchy}, let $\ga=\ga_1$ be the unique positive root of $g$. By the choice of coefficients in \eqref{ggg1} and in view of Lemma \ref{Ferguson}, the polynomial $g$ has all its other roots $\ga_2,\dots,\ga_m$ in the disc
$|z|<\ga$. Thus, the irreducible polynomial $g$ has only one root on the circle $|z|=\ga$, this implies that $\ell(g)=1$. So, $g$ is non-degenerate. Hence, none of the quotients $\ga_k/\ga_j$, where $k \ne j$, is a root of unity.  Since the $m\ell=n$ roots of $f$ are exactly $e^{2 \pi i (k-1)/\ell} \ga_j^{1/\ell}$, where $i=\sqrt{-1}$, $1 \le j \le m$ and $1 \le k \le \ell$, we must have $f \in S_{n,\ell}(H)$ if $f$ is irreducible.

Assume that $f$ is reducible and that its irreducible factor $f_1 \in \Z[X]$ has, say, $t \ge 1$ roots on the circle $|z|=\ga_1^{1/\ell}$ including the root $\ga_1^{1/\ell}$. Let $K$ be the splitting field of $f$ over $\Q$. Since $g$ is irreducible, for any $2\le j\le m$, there exists an automorphism of $\text{Gal}(K/\Q)$ that maps $\ga_1$ to $\ga_j$. It follows that $\ga_1^{1/\ell}$ has exactly $t$ conjugates in each set $e^{2 \pi i(k-1)/\ell} \ga_j^{1/\ell}$, where $i=\sqrt{-1}$, $j \ge 2$ is fixed and $k=1,\dots,\ell$.
Thus, as $\deg f_1<\deg f$, we have $1 \le t<\ell$ and the modulus of the product of all the conjugates of $\ga_1^{1/\ell}$ is equal to
$$|f_1(0)|=|\ga_1^{t/\ell} \ga_2^{t/\ell} \cdots \ga_m^{t/\ell}|=
|g(0)|^{t/\ell}=2^{t/\ell}(2b_m-1)^{t/\ell} \in \Z.$$
Now, it is easy to see that this is impossible, because the number $2^{t/\ell}(2b_m-1)^{t/\ell}$ is irrational in view of $t<\ell$.
\end{proof}

Now, we are ready to prove Theorem \ref{InH} which gives sharp bounds for the quantity $I_n(H)$.

\begin{proof}[Proof of Theorem~\ref{InH}]
Since $I_{n,\ell}(H)=0$ when $\ell$
does not divide $n$,
we have
$$
I_n(H)=\sum\limits_{\ell=2,\ell|n}^{n} I_{n,\ell}(H).
$$
The desired result follows immediately from Proposition~\ref{Inl}, since the largest contribution comes from the term $I_{n,p}(H)$, where $p$ is the smallest prime divisor of $n$.
\end{proof}

\section{Sharp bounds for $R_n(H)$}\label{RnHH}

In order to get an upper bound for $R_n(H)$, we need the following lemma, which may be of independent interest.

\begin{lemma}\label{linear}
For each integer $n \ge 4$, the number of polynomials, contained in $S_n(H)$ and having a linear factor, is bounded above by $O(H^{n-2})$.
\end{lemma}
\begin{proof}
Every polynomial $f$ we consider has the following form
$$
f=g(X)h(X)\in S_n(H),
$$
where $g(X)=X+a$ for some $a\in \Z$, and $h(X)$ is a monic integer polynomial of degree $n-1$. Since $M(f)=M(g)M(h)$, we have
$$
H(f)\ll H(g)H(h)\ll H(f).
$$
Note that the implied constants depend only on the degree $n$, and are independent of $f,g$ and $h$.
If $H(g)=k\le H$ (that is, $a=\pm k$ for $k \ge 2$, and $a=0, \pm 1$ for $k=1$), then we have $H(h)\ll H/k$ for each $k=1,\dots,H$.

Since $f$ is degenerate, $h$ is either degenerate or it has the linear factor $X-a$
(and so $a\ne 0$, and $h$ of degree $n-1$ is reducible).
In both cases,
either
by
Proposition~\ref{trivial} or by Corollary~\ref{reducible-1}, the number of such polynomials $h$ is $O((H/k)^{n-2})$. Notice that the constants implied in the symbols $\ll$ and $O$ depend only on $n$ and are independent of $k$ and $H$, so we can sum them up.
 Then, the number of such polynomials $f$ can be bounded above by
\begin{equation}\label{linear2}
\sum\limits_{k=1}^{H}(H/k)^{n-2}< H^{n-2} \sum\limits_{k=1}^{\infty}(1/k)^{n-2}=\zeta(n-2)H^{n-2}
\end{equation}
up to some constant, where $\zeta$ is the Riemann zeta function. Note that $\zeta(m) \le \zeta(2)=\pi^2/6$  for any integer $m \ge 2$. Since $n\ge 4$, we finally get the desired upper bound $O(H^{n-2})$.
\end{proof}

In fact, for $n\ge 4$, notice that any polynomial $(X^2-1)g(X)$, where $g$ is a monic polynomial with integer coefficients of degree $n-2$ and height $\le \lfloor H/2 \rfloor$, is contained in $S_n(H)$ and has a linear factor.
Hence, the number of reducible polynomials, contained in $S_n(H)$ and having a linear factor, is bounded below by $H^{n-2}$ up to some constant. Thus, combining with Lemma \ref{linear}, we get sharp bounds for the number of such polynomials.

Below, we first deduce an upper bound for $R_n(H)$, $n\ge 5$, and then deal with the exceptional cases $n=3,4$. These exceptional cases will be treated in a more general way that can be easily extended to count non-monic degenerate polynomials in Section \ref{DnH*}.

\begin{proposition}\label{Rn}
For $n\ge 5$, we have $R_n(H)\ll H^{n-2}$.
\end{proposition}
\begin{proof}
Since $n\ge 5$, by Lemma \ref{reducible}, the number of reducible polynomials $f\in S_n(H)$ whose irreducible factor with smallest degree has degree greater than $1$ is $O(H^{n-2})$. So, we only need to consider those reducible polynomials $f\in S_n(H)$ which have a linear factor. Then, the desired result follows directly from Lemma \ref{linear}.
\end{proof}

\begin{proposition}\label{R3}
We have $R_3(H)\ll H\log H$.
\end{proposition}

\begin{proof}
Note that any reducible polynomial $f\in S_3(H)$ must have the form
$$
f(X)=g(X)h(X)\in S_3(H),
$$
where $g(X)=X+a$ for some $a\in \Z$, and $h(X)$ is a quadratic degenerate monic polynomial. (In case $h$ is not degenerate it must be of the form $h(X)=(X-a)(X+b)$ with $a,b \in \Z$. Then, we can simply replace the pair $(g(X),h(X))=(X+a,(X-a)(X+b))$ by the pair $(g(X),h(X))=(X+b,(X-a)(X+a))$.) As in the proof of Lemma \ref{linear}, using Proposition~\ref{trivial}, we see that the number of such polynomials $f$ can be bounded above by
\begin{equation}\label{R33}
\sum\limits_{k=1}^{H}H/k=H\sum\limits_{k=1}^{H}1/k\ll H\log H
\end{equation}
up to some constant.
\end{proof}

\begin{proposition}\label{R4}
We have $R_4(H)\ll H^{2}$.
\end{proposition}
\begin{proof}
By Lemma \ref{linear}, the number of polynomials contained in $S_4(H)$ and having a linear factor is bounded above by $O(H^2)$. So, we only need to count those polynomials which can be factored into the product of two irreducible quadratic polynomials.
Let
$$
f(X)=g(X)h(X)\in S_4(H),
$$
where $g$ and $h$ are irreducible quadratic monic polynomials.

Suppose first that both $g$ and $h$ are non-degenerate with roots $\al,\al'$ and $\be,\be'$,
respectively. Since $f$ is degenerate, there exist a root, say $\al$, of $g$ and a root, say $\be$, of $h$ such that $\al/\be =\eta$ is a root of unity. Then, mapping $\al$ to
$\al'$ we obtain $\al'=\be' \eta'$ with a root of unity $\eta'$.
(If $\al'=\be \eta'$, then the quotient $\al/\al'=\eta/\eta'$ is a root of unity, which is not the case.) From $\al \al'=\be \be' \eta \eta'$ we deduce that
$\eta \eta'= \pm 1$, since $\al\al', \be\be' \in \Z$. Thus,
$g(X)=X^2+uX+v \in \Z[X]$ and $h(X)=X^2+wX \pm v \in \Z[X]$.
From $v^2 \le H$ and $\max(|uw|,|2v+uw|) \le H$, we derive that $|v| \le \sqrt{H}$
and $|uw| \le H+2\sqrt{H}$. Hence, for $H$ large enough, the number of such integer triplets $(u,w,v)$ is bounded above by
$O(H^{3/2}\log H)=O(H^2)$.

From now on, we assume that at least one of $g$ and $h$ is degenerate.
Without loss of generality, we assume that $g$ is degenerate. Here, we briefly follow the proof of Lemma \ref{linear}. Notice that $M(f)=M(g)M(h)$. Thus,
$$
H(f)\ll H(g)H(h) \ll H(f).
$$
If $H(g)=k\le H$, where $k \in \N$, then $H(h)\ll H/k$, so the number of such polynomials $h$ is $O((H/k)^2)$, and, as in the proof of  Proposition \ref{trivial}, there are finitely many choices of $g$ (since we choose one coefficient to be $\pm k$). Then, the number of such polynomials $f$ can be bounded above by
\begin{equation}\label{R44}
\sum\limits_{k=1}^{H}(H/k)^2< H^2\sum\limits_{k=1}^{\infty}1/k^2 =\zeta(2)H^2
\end{equation}
up to some constant.

Summarizing the above results, we obtain $R_4(H)\ll H^{2}$.
\end{proof}

Finally, as an application of the above results, we complete the proof of Theorem~\ref{RnH}.

\begin{proof}[Proof of Theorem~\ref{RnH}]
Evidently, each reducible polynomial $f\in S_2(H)$ must have the form $f(X)=X^2-a^2$ for some $a\in\N$ satisfying $a^2 \le H$. Hence, $R_2(H)=\lfloor \sqrt{H} \rfloor$.

Note that Propositions \ref{Rn}, \ref{R3} and \ref{R4} give the required upper bounds, so we only need to prove the lower bounds on $R_n(H)$ for $n \ge 3$.

Notice that, for any $a,b\in\N$, the polynomial $$f(X)=(X+a)(X^2+b)=X^3+aX^2+bX+ab$$ is reducible and degenerate. Since different pairs $(a,b)$ give different polynomials, and $f \in S_3(H)$ when $ab \le H$, we find that
$$
R_3(H) \ge \sum\limits_{k=1}^{H} \lfloor H/k \rfloor.
$$
By the properties of harmonic series, we obtain $R_3(H) \gg H\log H$.

Similarly, for $n\ge 4$, we consider
$$
f(X)=(X^2+1)g(X) \in\Z[X]
$$
with $g(X)=X^{n-2}+b_1X^{n-3}+\cdots+b_{n-2}\in\Z[X]$ whose coefficients satisfy
$1 \le b_1,\dots, b_{n-2} \le \lfloor H/2 \rfloor$.
Obviously, such $f$ is reducible, degenerate and of height at most $H$.
There are exactly $\lfloor H/2 \rfloor^{n-2}$ of such polynomials.
Consequently, $R_n(H) \gg H^{n-2}$.
\end{proof}

\section{Proof of Theorem~\ref{mainn}}\label{DnH*}

In order to prove Theorem \ref{mainn}, we first state the following sharp bounds for $I_n^{*}(H)$ and $R_n^{*}(H)$, respectively.

\begin{theorem}\label{InH*}
Let $n\ge 2$ and $H \ge 1$ be two integers, and let $p$ be the smallest prime divisor of $n$. Then
$$H^{1+n/p}\ll I_n^{*}(H) \ll H^{1+n/p}.$$
\end{theorem}

\begin{theorem}\label{RnH*}
For integers $n\ge 2$ and $H\ge 2$, we have
$$H \log H \ll R_2^{*}(H) \ll H \log H,$$
$$H^2\log H \ll R_3^{*}(H) \ll H^2\log H,$$
$$\qquad\qquad H^{n-1} \ll R_n^{*}(H) \ll H^{n-1} \quad (n\ge 4).$$
\end{theorem}

We remark that, unlike in all formulas of Theorems~\ref{InH*} and \ref{RnH*}, the size of $R_2^{*}(H)$ is not
related to $R_2(H)=\lfloor \sqrt{H} \rfloor$ by adding an extra power of $H$.

In the following, we want to briefly follow the deductions in the previous sections to prove Theorems \ref{InH*} and \ref{RnH*}. Then, Theorem \ref{mainn} follows directly.

\begin{proof}[Proof of Theorem \ref{InH*}]
We define $I_{n,\ell}^{*}(H)$ as the number of irreducible polynomials $f\in S_n^{*}(H)$ such that $\ell(f)=\ell$, and write $s(f)$ as $s$, where $\ell(f)$ and $s(f)$ have been defined in \eqref{equivalence}.
Obviously,
$$
I_n^{*}(H)=\sum\limits_{\ell=2,\ell|n}^{n} I_{n,\ell}^{*}(H).
$$
As in Section \ref{InHH}, following the proof of Proposition \ref{Inl}, we replace the $g(X)$ given in \eqref{ggg} by
$$
g(X)=a_0(X-\be_1)\cdots(X-\be_s)\in \Z[X],
$$
where $a_0$ is the leading coefficient of the corresponding polynomial $f$.
Here, $g \in \Z[X]$, since any product $a_0 \prod_{j \in I} \al_j$, where
$\al_1,\dots,\al_n$ are the roots of $f$ and
$I$ is a subset of the set $\{1,2,\dots,n\}$, is an algebraic integer.
Now, as in Proposition~\ref{Inl}, this gives the upper bound $I_{n,\ell}^{*}(H) \ll H^{1+n/\ell}$.

For the lower bound,
instead of \eqref{ggg1}, we consider
$$g(X)=(2b_0-1)X^m-2b_1X^{m-1}-\dots-2b_{m-1}X-2(2b_m-1) \in \Z[X],$$
where $1 \le b_0,b_1,\dots, b_{m-1} \le \lfloor H/2 \rfloor$,
$ 1 \le b_m \le  \lfloor H/4 \rfloor$.
There are at least $\lfloor H/4\rfloor^{m+1}$
of such polynomials $g$ for $H \ge 4$.
As before, by Lemmas \ref{Ferguson} and \ref{Cauchy}, let $\ga$ be the unique positive root of $g$, its other roots $\ga_2,\dots,\ga_m$ lie in the disc $|z|<\ga$. Thus, the irreducible polynomial $g$ has only one root on the circle $|z|=\ga$, this implies that $\ell(g)=1$. So, $g$ is non-degenerate.
 Since $m=n/\ell$,
for $\ell|n$ (and $\ell \ge 2$), as in the proof of Proposition~\ref{Inl},
we deduce that
$$
H^{1+n/\ell}\ll I_{n,\ell}^{*}(H) \ll H^{1+n/\ell},
$$
which implies Theorem \ref{InH*}.
\end{proof}

\begin{proof}[Proof of Theorem~\ref{RnH*}]
To get an analogue of Lemma \ref{linear}, we note that the form of $f$ of degree $n \ge 4$ becomes
$$
f(X)=g(X)h(X)\in S_n^{*}(H),
$$
where $g(X)=a_0X+a_1$ , and $h(X)$
is an integer polynomial of degree $n-1$. Fix an integer $k$ with $1\le k\le H$, let $H(g)=k$, then the number of such polynomials $g$ with $H(g)=k$ is $O(k)$, and as before, we have $H(h)\ll H/k$.
Since $f$ is degenerate, $h$ is either degenerate or it has the linear factor $a_0X-a_1$
(and so $a_1\ne 0$, and $h$ of degree $n-1$ is reducible).
In both cases, either by
Proposition~\ref{trivial} or, this time, by Lemma~\ref{reducible2}, the number of such polynomials $h$ is $O((H/k)^{n-1})$.
Then, the bound in \eqref{linear2} becomes
$$
\sum\limits_{k=1}^{H}k(H/k)^{n-1}< H^{n-1} \sum\limits_{k=1}^{\infty}(1/k)^{n-2}=\zeta(n-2)H^{n-1},
$$
which implies the desired upper bound $O(H^{n-1})$ for $n\ge 4$. Then as in the proof of Proposition \ref{Rn} (this time, applying Lemma \ref{reducible2}), we get the analogue of Proposition~\ref{Rn}
$$
R_n^{*}(H)\ll H^{n-1} \quad \textrm{for each} \quad n\ge 5.
$$

Following the proof of Proposition \ref{R3}, we replace the form of $f$ by $f=g(X)h(X)\in S_3^{*}(H),$ where $g(X)=a_0X+a_1$, and $h(X)$ is a quadratic degenerate polynomial. Note that
the number of such polynomials $g$ with $H(g)=k$ is $O(k)$  for every $k \in \N$. Using Proposition~\ref{trivial}, we see that
the bound in \eqref{R33} becomes
$$
\sum\limits_{k=1}^{H}k(H/k)^2=H^2\sum\limits_{k=1}^{H}1/k\ll H^2\log H.
$$ Consequently,
$$R_3^{*}(H)\ll H^2\log H,$$
which is the analogue of Proposition~\ref{R3}.

To get an upper bound for $R_4^{*}(H)$, we should change the form of $f$ in the proof of Proposition \ref{R4} to be
$$
f(X)=g(X)h(X)\in S_4^{*}(H),
$$
where $g$ and $h$ are irreducible quadratic polynomials.

Assume that both $g$ and $h$ are non-degenerate. We also borrow the notation from the proof of Proposition \ref{R4}. Since $f$ is degenerate, there exist a root $\al$ of $g$ and a root $\be$ of $h$ such that $\al/\be =\eta$ is a root of unity. Multiplying with
$\al'=\be'\eta'$, we find that
$\al \al'=\be \be' \eta \eta'$. Hence,
$\eta \eta'= \pm 1$, since $\al\al', \be\be' \in \Q$. Thus,
$g(X)=tX^2+uX+v \in \Z[X]$ and $h(X)=t_1 X^2+wX +v_1 \in \Z[X]$,
where
\begin{equation}\label{vvv}
|v/t|=|v_1/t_1|.
\end{equation}
Without restriction of generality, we may assume that $g$ is primitive, namely, no prime $p$ divides all three coefficients $t,u,v$ of $g$.

Suppose first that both pairs $\al,\al'$ and $\be,\be'$ are real. Then
$\eta=\al/\be=-1$ and $\eta'=\al'/\be'=-1$. So, notice that $g$ is primitive, if $g(X)=tX^2+uX+v \in \Z[X]$, then
$h(X)=c(tX^2-uX+v)$ with some non-zero $c \in \Z$. Since the Mahler measures of $g$ and $h/c$ are the same
 and
$$
M(g)M(h/c)\ll H(g)H(h/c)\ll H(gh/c)=H(f)/c,
 $$,
 we have $H(g) \ll M(g) \ll \sqrt{H/|c|}.$ This gives the following upper bound for the number of $(t,u,v,c) \in \Z^4$
$$\sum_{|c|=1}^{H} |c| (H/|c|)^{3/2}= H^{3/2} \sum_{|c|=1}^{H}
|c|^{-1/2} \ll H^2$$
up to some constant.

Suppose next that at least one pair of conjugates, say $\be,\be'$, are complex (non-real) numbers. Then they are complex conjugates. Moreover, the conjugates $\al$ and $\al'$ have the same modulus.
Since $\al/\al' \ne -1$ ($g$ is non-degenerate), the numbers $\al,\al'$ are also complex conjugate numbers. From $$|\al|=|\al'|=|\be|=|\be'|=R>0$$ and
$$g(X)h(X)=(tX^2+uX+v)(t_1X^2+wX+v_1),$$
we find that $|vv_1| \le H$ and $|tt_1| \le H$.
Consider irreducible rational fraction
$a/b$ with $a, b \in \N$ satisfying $ab \le H$.
Let $(v_1,v,t_1,t) \in \Z^4$ be a vector satisfying $|v_1|/|v|=|t_1|/|t|=a/b$
(see \eqref{vvv}), where $1 \le ab, |tt_1|, |vv_1| \le H$. Then, for some
$s,s_1 \in \N$, we must have $$|v_1|=sa, \>\>\> |v|=sb, \>\>\> |t_1|=s_1 a, \>\>\> |t|=s_1 b,$$ so that $s,s_1 \le \sqrt{H/ab}$.  This gives at most $ss_1 \le H/ab$ vectors
$(|v_1|,|v|,|t_1|,|t|) \in \N^4$ corresponding to the fixed fraction $a/b$.
 Therefore, the number of such vectors
$(v_1,v,t_1,t) \in \Z^4$ is bounded from above (up to some constant) by
$$\sum_{a=1}^H \sum_{b=1}^{\lfloor H/a \rfloor} \frac{H}{ab} =O(H(\log H)^2).$$

Also, $u=-t(\al+\al')$, so $$|u| \le |t|(|\al|+|\al'|) = 2|t|R=2|t| \sqrt{|v|/t}=2\sqrt{|vt|}.$$ Similarly, $|w| \le 2\sqrt{|v_1t_1|}$. Thus, $|uw| \le 4 \sqrt{|vtv_1t_1|} \le 4H$. Clearly, there are $O(H \log H)$ of such pairs $(u,w) \in \Z^2$. Therefore, in case $\be,\be'$ are complex (non-real) numbers, we bound the number of vectors $(v_1,v,t_1,t,u,w) \in \Z^6$ from above by $O(H^2 (\log H)^3)$.

Summarizing the above results, when none of the quadratic factors $g$ and $h$ is degenerate,
the bound is $O(H^2 (\log H)^3)$, which is better than the required bound $O(H^3)$.

It remains to consider the case when at least one of the factors $g,h$, say $g$, is degenerate.
Then, in view of Proposition~\ref{trivial}, the bound in \eqref{R44} becomes
$$
\sum\limits_{k=1}^{H}k(H/k)^3< H^3\sum\limits_{k=1}^{\infty}1/k^2 =\zeta(2)H^3.
$$
This yields the required analogue of Proposition~\ref{R4}, namely,
$$R_4^{*}(H)\ll H^3.$$

Now, we will complete the proof of Theorem \ref{RnH*}.
It is easy to see that each reducible polynomial $f\in S_2^{*}(H)$ must be of the form $$f(X)=\pm c(a^2X^2-b^2)$$ for some $a,b\in\N$ and some square-free $c \in \N$. From $1 \le c \le H$ and $1 \le a, b \le \lfloor \sqrt{H/c} \rfloor$, we obtain
\begin{equation}\label{R2*}
R_2^{*}(H)=2 \sum_{c=1, \> c \>-\>\text{square-free}}^H \lfloor \sqrt{H/c} \rfloor^2.
\end{equation}
Using the fact that the sum $\sum' 1/c$
taken over square-free $c$ in the range $1\le c \le H$ satisfies $\log H \ll \sum' 1/c \ll \log H$, we obtain the required bounds $H\log H \ll R_2^{*}(H) \ll H\log H$.

To get a lower bound for $R_3^{*}(H)$ observe that, for any  $a,b,c,d\in\N$, the polynomial $$f(X)=(aX+b)(cX^2+d)=acX^3+bcX^2+adX+bd$$ is reducible and degenerate. Moreover, if $c$ and $d$ are coprime, then different integer vectors $(a,b,c,d) \in \Z^4$ give different $f$. By counting all the vectors
$(a,b,c,d) \in \N^4$ satisfying $1\le a\le b\le H, 1 \le c,d \le H/b$, we find
that there are exactly
$$
\sum\limits_{b=1}^{H} b \lfloor H/b \rfloor^2 \gg H^2\log H
$$
of them. The probability that two integers are coprime is $1/\zeta(2)$. So, taking into account only coprime pairs $(c,d)$ in the above sum, we will still get $\gg H^2 \log H$
of such vectors $(a,b,c,d)$ (with another constant implied in $\gg$ and not depending on $H$).
 Consequently, we obtain $R_3^{*}(H) \gg H^2\log H$.

Finally, for $n\ge 4$, let us consider
$$
f(X)=(X^2+1)g(X) \in\Z[X],
$$
where $g(X)=b_0X^{n-2}+b_1X^{n-3}+\cdots+b_{n-2}\in\Z[X]$ with coefficients $1 \le b_0,\dots,b_{n-2} \le \lfloor H/2\rfloor$. Obviously, such $f$ are all reducible, degenerate and have height at most $H$.
There are exactly $\lfloor H/2\rfloor^{n-1}$ of such vectors  $(b_0,\dots,b_{n-2})$ and each of them gives a different polynomial $f$.
This yields $R_n^{*}(H) \gg H^{n-1}$, and completes
the proof of Theorem~\ref{RnH*}.
\end{proof}

\section{Some asymptotic formulas}\label{comment}

It would be of interest to find out whether there exist asymptotic (or exact) formulas for all the quantities we consider, especially for $D_n(H)$ and $D_n^*(H)$. In general, this seems to be a difficult problem, even though
it is very likely that both the limits $\lim_{H\to \infty}D_n(H)/H^{n-2}$ and $\lim_{H\to \infty}D_n^*(H)/H^{n-1}$ exist for any $n\ge 4$. However, as the results of \cite{mnn} indicate, one should not expect any nice formulas for the constants involved in the main terms of $R_n(H)$ and $R_n^*(H)$ for
$n \ge 4$.

Here, we will consider some special cases and let the reader taste where the difficulties may come for other values of $n$.

\begin{theorem}\label{limit2}
For each $H \ge 1$ we have $R_2(H)=\lfloor \sqrt{H} \rfloor$,
$$I_2(H)=2H+\lfloor \sqrt{H} \rfloor + 2\lfloor \sqrt{H/2} \rfloor + 2\lfloor \sqrt{H/3}\rfloor,$$
$$D_2(H)=2H+2\lfloor \sqrt{H} \rfloor + 2\lfloor \sqrt{H/2} \rfloor + 2\lfloor \sqrt{H/3}\rfloor.$$
\end{theorem}
\begin{proof}
By Theorem~\ref{RnH} and $D_2(H)=I_2(H)+R_2(H)$, we see that it suffices to prove the formula for $I_2(H)$.

For any irreducible polynomial $f\in S_2(H)$, let $\al \ne \be$ be its roots. Then, the degree of the root of unity $\al/\be$ must be at most 2, so $\al/\be=-1,\pm i, (-1\pm i\sqrt{3})/2$ or $(1\pm i\sqrt{3})/2$, where $i=\sqrt{-1}$.

If $\al/\be=-1$, we must have $f(X)=X^2-a$, where $a$ is not a perfect square and $1\le |a|\le H$. So, the number of such polynomials $f$ is $2H-\lfloor \sqrt{H} \rfloor$.

If $\al/\be=\pm i$, we have $f(X)=X^2+ 2aX+2a^2$, where $a\in \Z$ and $1\le |a|\le \lfloor \sqrt{H/2} \rfloor$. Note that each $a$ corresponds to a different $f$. Clearly, the number of these polynomials $f$ is $2\lfloor \sqrt{H/2} \rfloor$.

If $\al/\be=(-1\pm i\sqrt{3})/2$, we have $f(X)=X^2+ aX+a^2$, where $a\in \Z, 1\le |a|\le \lfloor \sqrt{H} \rfloor$.  The number of these polynomials $f$ is $2\lfloor \sqrt{H} \rfloor$.

Finally, if $\al/\be=(1\pm i\sqrt{3})/2$, we have $f(X)=X^2+ 3aX+3a^2$, where $a\in \Z$ and $1\le |a|\le \lfloor \sqrt{H/3} \rfloor$. The number of such polynomials $f$ is $2\lfloor \sqrt{H/3} \rfloor$.

Summarizing the above results, we obtain exactly $2H+\lfloor \sqrt{H} \rfloor + 2\lfloor \sqrt{H/2} \rfloor + 2\lfloor \sqrt{H/3}\rfloor$ irreducible polynomials contained in $S_2(H)$.
\end{proof}

\begin{theorem}\label{limit3}
We have $\lim_{H\to \infty}D_3(H)/(H\log H)=4$.
\end{theorem}
\begin{proof}
By Theorems \ref{InH} and \ref{RnH}, it is equivalent to prove that
$$\lim_{H\to \infty}R_3(H)/(H\log H)=4.$$

As in the proof of Proposition \ref{R3}, any reducible polynomial $f\in S_3(H)$ must have the form
$$
f(X)=g(X)h(X)\in S_3(H),
$$
where $g(X)=X+a$ for some $a\in \Z$, and $h(X)$ is a quadratic degenerate monic polynomial.

Let $\al,\be$ be the two distinct roots of $h$. Then, as above, $\al/\be=-1,\pm i, (-1\pm i\sqrt{3})/2$ or $(1\pm i\sqrt{3})/2$. If $\al/\be=-1$, we have $f(X)=(X+a)(X^2+b)$, where $1\le |ab|\le H$. Note that each pair $(a,b)$ corresponds to a different $f$. So, the number of these polynomials $f$ is exactly
\begin{equation}\label{tres}
2\sum\limits_{|b|=1}^{H}\left(2\lfloor H/|b| \rfloor+1\right).
\end{equation}

If $\al/\be$ takes other values, then $h$ is irreducible. Using the forms of polynomials just considered in the proof of Theorem \ref{limit2}, we can see that the number of such polynomials $f$ is bounded above by the number of integer pairs $(a,b)$ with $ab^2\le H$ (up to some constant), which gives the term $O(H)$.
Hence, as the main term in \eqref{tres} is $4H\log H$, we obtain $\lim_{H\to \infty}R_3(H)/(H\log H)=4$.
\end{proof}

\begin{theorem}\label{limit2*}
We have
$$\lim_{H\to \infty}I_2^*(H)/H^2=4,$$
$$\lim_{H\to \infty}R_2^*(H)/(H\log H)=2/\zeta(2)=12/\pi^2.$$
In particular, we have $\lim_{H\to \infty}D_2^*(H)/H^2=4$.
\end{theorem}
\begin{proof}
Let $f\in S_2^*(H)$ be an irreducible polynomial, and let $\al, \be$ be its two distinct roots. Then, as before, $\al/\be=-1,\pm i, (-1\pm i\sqrt{3})/2$ or $(1\pm i\sqrt{3})/2$. Applying the same argument as in the proof of Theorem \ref{limit2}, we can see that the main contribution to $I_2^*(H)$ comes from those irreducible polynomials $f$ with $\al/\be=-1$, that is, $f=aX^2+b$, where $1\le |a|,|b|\le H, (a,-b)\ne \pm c(a^\prime,b^\prime),c\in \N$ is square-free, and $a^\prime, b^\prime$ are perfect squares. Thus,
$$\lim_{H\to \infty}I_2^*(H)/H^2=4.$$

Recall that, by \eqref{R2*},
\begin{align*}
R_2^{*}(H)&=2 \sum_{k=1, \> \textrm{$k$ is square-free}}^H \lfloor \sqrt{H/k} \rfloor^2.
\end{align*}
Since every positive integer can be uniquely expressed as the product of a perfect square and a square-free integer, we have
$$
\lim_{H\to \infty}\frac{\sum_{k=1, \> \textrm{$k$ is square-free}}^{H}1/k}{\sum_{k=1}^{H}1/k}=\frac{1}{\sum_{j=1}^{\infty} 1/j^2} =\frac{1}{\zeta(2)}.
$$
Then, we obtain $\lim_{H\to \infty}R_2^*(H)/(H\log H)=2/\zeta(2)$.
The last statement of the theorem follows from $D_2^{*}(H)=I_2^{*}(H)+R_2^{*}(H)$.
\end{proof}

\begin{theorem}\label{limit3*}
We have $\lim_{H\to \infty}D_3^*(H)/(H^2\log H)=96/\pi^2$.
\end{theorem}
\begin{proof}
By Theorems \ref{InH*} and \ref{RnH*}, it is equivalent to prove that
$$\lim_{H\to \infty}R_3^*(H)/(H^2\log H)=96/\pi^2.$$

Recall that, any reducible polynomial $f\in S_3^*(H)$ must have the form
$$
f(X)=g(X)h(X)\in S_3^*(H),
$$
where $g(X)=aX+b$ for some $a,b\in \Z$, and $h(X)$ is a quadratic degenerate polynomial.

As before, the main contribution to $R_3^*(H)$ comes from those polynomials $f$ such that the quotient of the two roots of $h$ is $-1$, that is,
$f(X)=(aX+b)(cX^2+d)$ with $c>0$ and $d \ne 0$. Without loss of generality,
we may assume that $(c,d)=1$, by putting their common factor into the linear factor $aX+b$.
Note that the number of these polynomials $f$ with $b=0$ is $O(H^2)$, so in the sequel we assume that $b\ne 0$.

Now, as $1\le |a|, |b| \le \min(H/c, H/|d|)$ we consider two cases $1\le c<|d|$ and $c>|d|\ge 1$
(there are two other cases with $c=|d|=1$, but they only give $O(H^2)$ of such polynomials).
In the first case, $\min(H/c, H/|d|)=H/|d|$, the number of vectors $(a,b,c,d)$ is
\begin{align*}
8 \sum_{|d|=2}^H \lfloor H/|d| \rfloor ^2 \sum_{c=1, (c,|d|)=1}^{|d|-1} 1& \sim  8H^2 \sum_{|d|=2}^H \varphi(|d|)/|d|^2\\
&= -8H^2+8H^2 \sum_{|d|=1}^H \varphi(|d|)/|d|^2\\
&\sim (48/\pi^2) H^2 \log H,
\end{align*}
where $\varphi$ is Euler's totient function, $\sim$ is the standard asymptotic no\-ta\-tion.
(Here, we use the standard formula $\sum_{j=1}^H \varphi(j)/j^2\sim (6/\pi^2)\log H$, whereas the factor
8 comes, since we have three moduli $|a|,|b|,|d|$.)

By exactly the same argument, we can see that the same contribution $(48/\pi^2) H^2 \log H$ comes from the case $c>|d|\ge 1$. Therefore,
$\lim_{H\to \infty}R_3^*(H)/(H^2\log H)=96/\pi^2$, as claimed.
\end{proof}

\section*{Acknowledgements}
The authors would like to thank Igor E. Shparlinski for introducing them
into this topic and for providing Proposition \ref{trivial} which was their
starting point, and also for his valuable comments on an early version of this paper. The research of A.~D. was supported by the Research Council
of Lithuania Grant MIP-068/2013/LSS-110000-740. The research of M.~S. was supported by the Australian
Research Council Grant DP130100237.
They also want to thank the referee for careful reading and useful comments.

\end{document}